\documentclass[12pt]{article}
\usepackage{amsmath,amssymb,amsthm}
\usepackage{comment}
\usepackage{color}
\numberwithin{equation}{section}
\newtheorem{thm}{Theorem}[section]
\newtheorem{prop}[thm]{Proposition}

\newtheorem{exam}[thm]{Example}
\newtheorem{rem}[thm]{Remark}
\newtheorem{lem}[thm]{Lemma}

\newtheorem{assum}[thm]{Assumption}

\def\D{{\cal D}}
\def\E{{\cal E}}
\def\F{{\cal F}}

\def\R{{\mathbb R}}
\def\Ss{{\mathbb S}}

\def\d{{\rm d}}

\allowdisplaybreaks

\title{Compactness of semigroups generated by
symmetric non-local Dirichlet forms with
unbounded coefficients}
\author{Yuichi Shiozawa\thanks{Department of Mathematics,
Graduate School of Science,
Osaka University, Toyonaka, Osaka, 560-0043,
Japan; \texttt{shiozawa@math.sci.osaka-u.ac.jp}}\qquad
Jian Wang\thanks{College of Mathematics and Informatics  \&
Fujian Key Laboratory of Mathematical Analysis and Applications (FJKLMAA)  \& Center for Applied Mathematics of Fujian Province (FJNU),
Fujian Normal University, Fuzhou, 350007, P.R. China; \texttt{jianwang@fjnu.edu.cn}}}

\begin{document}
\maketitle
\begin{abstract}
Let $(\E,\F)$ be a symmetric non-local Dirichlet from
with unbounded coefficient on $L^2(\R^d;\d x)$ defined by
$$\E(f,g)=\iint_{\R^d\times \R^d} (f(y)-f(x))(g(x)-g(y)){W(x,y)}\, J(x,\d y)\,\d x,
\quad f,g\in \F,$$
where $J(x,\d y)$ is regarded as the jumping kernel
for a pure-jump symmetric L\'evy-type process
with bounded coefficients,
and $W(x,y)$ is seen as a weighted (unbounded) function.
We establish sharp criteria
for compactness and non-compactness of
the associated
Markovian semigroup $(P_t)_{t\ge0}$ on $L^2(\R^d;\d x)$.
In particular, we prove that if $J(x,\d y)=|x-y|^{-d-\alpha}\,\d y$ with $\alpha\in (0,2)$, and
$$W(x,y)=
\begin{cases}
(1+|x|)^p+(1+|y|)^p, \ &  |x-y|< 1  \\
(1+|x|)^q+(1+|y|)^q, \ &  |x-y|\geq 1
\end{cases}$$ with  $p\in [0,\infty)$ and $q\in [0,\alpha)$,
then $(P_t)_{t\ge0}$ is compact, if and only if $p>2$.
This indicates that the compactness of $(\E,\F)$ heavily depends
on the growth of the weighted function $W(x,y)$ only for $|x-y|<1$.
Our approach is based on establishing
the
essential super Poincar\'e inequality
for $(\E,\F)$. Our
general results work
 even if the jumping kernel $J(x,\d y)$ is degenerate or  is singular
 with respect to the Lebesgue measure.
\end{abstract}

\section{Introduction}

It is an important research subject in functional analysis and mathematical physics to reveal
the spectral structure of a self-adjoint operator $L$
generating a Markovian semigroup $(P_t)_{t\ge0}:=(e^{-tL})_{t\ge0}$.
In view of the study of symmetric Markov processes,
we can deduce asymptotic properties, in particular, ergodic properties,
of transition semigroups by using the spectral structure.
To do so,
we would like to find conditions for compactness of $(P_t)_{t\ge0}$
 because it is equivalent to
the essential spectrum of $L$ being empty (see, e.g., \cite[Theorem 0.3.9(ii)]{W05}).
In fact, if this is the case
for the symmetric semigroup $(P_t)_{t\ge0}$ on $L^2(\R^d;\d x)$,
then there exist
eigenvalues $\{\lambda_n\}_{n\ge1}$
and the corresponding eigenfunctions $\{\varphi_n\}_{n\ge1}\subset L^2(\R^d;\d x)$
such that $P_tf=\sum_{n=1}^{\infty}
e^{-\lambda_n t}\langle\varphi_n,f\rangle_{L^2(\R^d;\d x)}\varphi_n$
for $f\in L^2(\R^d;\d x)$ and $t>0$. The readers are referred to
\cite{Wu, Ta} for several equivalent conditions for the compactness of symmetric Markov semigroups.

If $L$  is a (symmetric) second order elliptic operator of
the form
$\sum_{i,j=1}^{d} \frac{\partial }{\partial x_i}(a_{ij}(x)\frac{\partial}{\partial x_j})$
on $L^2(\R^d;\d x)$,
then the spectral structure of
$(P_t)_{t\ge0}$ is characterized in terms of
the growth order of the coefficient $\{a_{ij}(x)\}_{1\le i\le d}$ at infinity
(see,\ e.g., \cite{D85, D90,P88, P96}).
More precisely, let $\{a_{ij}(x)\}_{1\leq i,j\leq d}$ be a
$C^\infty$ function on $\R^d$ with values
in the set of positive symmetric matrices,
and $\E^L$ the quadratic form on $C_c^{\infty}(\R^d)$ defined by
\begin{equation}\label{eq:elliptic}
\E^L(f,g)=\frac{1}{2}
\int_{\R^d}
\sum_{i,j=1}^d a_{ij}(x)
\frac{\partial f}{\partial x_i}(x)\frac{\partial g}{\partial x_j}(x)
\,\d x.
\end{equation}
Here $C_c^{\infty}(\R^d)$ is the totality of smooth functions on $\R^d$
with compact support.
For simplicity, we assume that
there exist positive constants $\lambda$, $\Lambda$  and $p$
such that for any $x, \xi\in \R^d$,
\begin{equation}\label{eq:coeff}
\lambda(1+|x|)^p|\xi|^2
\leq \sum_{i,j=1}^da_{ij}(x)\xi_i\xi_j
\leq \Lambda (1+|x|)^p|\xi|^2.
\end{equation}
Then, the closure of the quadratic form $(\E^L,C_c^{\infty}(\R^d))$
is a regular Dirichlet form on $L^2(\R^d;\d x)$
generating a Markovian semigroup
(see e.g., \cite[Chapter 1 and  Section 3.1]{FOT11}).
In particular, this semigroup is compact, if and only if  $p>2$
(see \cite[Theorem 4.2 and Corollary 4.4]{D85} and Remark \ref{rem:diff} below).

The purpose of this paper is to establish sharp criteria
for compactness and non-compactness of the semigroups
associated
with a large class of non-local Dirichlet forms on $L^2(\R^d;\d x)$
with unbounded coefficients. To highlight the novelty of our contribution,
we present the assertions for stable-like Dirichlet forms
with unbounded coefficients in this section,
which is a special case of general results
(see Theorems \ref{Thm2.1} and \ref{thm:cpt} below).

For $d \ge1$ and $\alpha\in (0,2)$, the fractional Laplacian is defined by
$$-(-\Delta)^{\alpha/2}f(x)
:=\lim_{\varepsilon\to0}\int_{\{|x-y|>\varepsilon\}}(f(y)-f(x))\frac{C_{d,\alpha}}{|x-y|^{d+\alpha}}\,\d y,$$
where $C_{d,\alpha}$ is a positive constant depending on $d$ and $\alpha$.
It is known that the essential spectrum of
the operator $(-\Delta)^{\alpha/2}$ is $[0,\infty)$.
Equivalently, the associated Markovian semigroup is not compact.
On the contrary, this semigroup has nice analytical properties;
in particular, it has a strictly positive density function with respect to the Lebesgue measure
(see, e.g.,\,\cite{CK03}).
Therefore, according to \cite[Theorem 0.3.9 and Theorem 3.2.1]{W05},
for any strictly positive function $\psi\in L^2(\R^d;\d x)$,
there exists no decreasing function $\beta: (0,\infty)\to (0,\infty)$
such that the following essential super Poincar\'e inequality holds
for all $r>0$ and $f\in C_c^\infty(\R^d)$:
\begin{equation}\label{e:poin-1}
\int_{\R^d} f(x)^2\,\d x\le r \E^0(f,f)
+\beta(r)\left(\int_{\R^d} |f(x)|\psi(x)\,\d x\right)^2.
\end{equation}
Here $\E^0(f,f)$ is the bilinear form associated with the fractional Laplacian $-(-\Delta)^{\alpha/2}$,
i.e., for any $f,g\in C_c^\infty(\R^d)$,
$$\E^0(f,g)=\langle (-\Delta)^{\alpha/2} f,g\rangle_{L^2(\R^d;\d x)}
=\frac{1}{2}\iint_{\R^d\times \R^d} (f(y)-f(x))(g(x)-g(y))\frac{C_{d,\alpha}}{|x-y|^{d+\alpha}}\,\d x\,\d y.$$
In the theory of stochastic processes, the fractional Laplacian $-(-\Delta)^{\alpha/2}$ is the infinitesimal generator of the rotationally symmetric $\alpha$-stable process, and so  $\E^0(f,g)$ is called the $\alpha$-stable Dirichlet form in the literature.

In order for the validity of the inequality \eqref{e:poin-1},
one reasonable way is to enlarge the bilinear form $\E^0(f,f)$ at the right hand side of \eqref{e:poin-1};
that is, instead of $\E^0(f,g)$,
we will consider the following $\alpha$-stable like Dirichlet form
with unbounded coefficient:
\begin{equation}\label{e:DF}
\E(f,g):= \iint_{\R^d\times \R^d} (f(x)-f(y))(g(x)-g(y))\frac{W(x,y)}{|x-y|^{d+\alpha}}\,\d x\,\d y,\end{equation}
where
$W(x,y)$ is a strictly positive, symmetric and unbounded measurable function on $\R^d\times \R^d$.
We can regard \eqref{e:DF} as a non-local analogue of \eqref{eq:elliptic}.
If \eqref{e:poin-1} holds with $\E^0(f,f)$ replaced by $\E(f,f)$,
then we can immediately get the compactness of the semigroup associated with the bilinear form $\E(f,g)$.
This explains the motivation of our paper.

Let $(\E,\D(\E))$ be a quadratic form on $L^2(\R^d; \d x)$ such that
$\E$ is as in \eqref{e:DF}, and
 $\D(\E)=\left\{f\in L^2(\R^d;\d x): \E(f,f)<\infty\right\}.$  We assume that
\begin{equation}\label{e:cor-1}x\mapsto \int_{\R^d} (1\wedge |x-y|^2)\frac{W(x,y)}{|x-y|^{d+\alpha}}\,\d y\in L^1_{{\rm loc}}(\R^d;\d x).
\end{equation}
It holds that $C_c^\infty(\R^d)\subset \D(\E)$.
Let $\|\cdot\|_{\E_1}$ be the norm on $\D(\E)$ defined by
 $\|f\|_{\E_1}=\big(\E(f,f)+\|f\|_{L^2(\R^d;\d x)}^2\big)^{1/2},$  and
 $\F:=\overline{C_c^{\infty}(\R^d)}^{\|\cdot\|_{\E_1}}.$ Then,
$(\E,\F)$ is a regular Dirichlet form on $L^2(\R^d;\d x)$. Denote by $(P_t)_{t\ge0}$  the associated Markovian semigroup.
We have the following statement.

\begin{thm}\label{T:main}
Let $W(x,y)$ be a Borel measurable function
on $\R^d\times \R^d$ defined by
\begin{equation}\label{eq:funct-w}
W(x,y)=(U_1(x)+U_1(y)){\bf 1}_{\{|x-y|<1\}}
+(U_2(x)+U_2(y)){\bf 1}_{\{|x-y|\geq 1\}},
\end{equation}
where
 $U_i(x)$ $(i=1,2)$ is a nonnegative locally bounded function on $\R^d$
such that for some $c_1>0$ and $q\in [0,\alpha)$,
\begin{equation}\label{e:cond000}
U_2(x)\leq c_1(1+|x|)^q,\,\quad x\in \R^d.
\end{equation}
Then, we have
\begin{itemize}
\item[{\rm(i)}]
if $\inf_{x\in \R^d} U_1(x)>0$ and
$$
\liminf_{|x|\rightarrow\infty}\frac{U_1(x)}{|x|^2}=\infty,
$$ then  $(P_t)_{t\geq 0}$ is compact on $L^2(\R^d;\d x)${\rm ;}
\item[{\rm(ii)}] if
$$
\limsup_{|x|\rightarrow\infty}\frac{U_1(x)}{|x|^2}<\infty,
$$
then $(P_t)_{t\geq 0}$ is not compact on $L^2(\R^d;\d x)$.
\end{itemize}
\end{thm}
\eqref{e:cond000} along with the local boundedness of $U_i$ $(i=1,2)$ implies that \eqref{e:cor-1} holds.
As a direct consequence of Theorem \ref{T:main},
by letting
\begin{equation}\label{e:DF--}W(x,y)=
\begin{cases}
(1+|x|)^p+(1+|y|)^p, \ &  |x-y|< 1  \\
(1+|x|)^q+(1+|y|)^q, \ &  |x-y|\geq 1
\end{cases}\end{equation}
with $p\in [0,\infty)$ and $q\in [0,\alpha)$,
the associated semigroup $(P_t)_{t\geq 0}$ is compact, if and only if,
$p> 2.$

We will make a few comments on Theorem \ref{T:main}.

\begin{itemize}
\item[(i)]
Theorem \ref{T:main} indicates that the compactness of the semigroup associated with the Dirichlet form $(\E,\F)$ given by \eqref{e:DF} heavily depends on the growth of the weighted function $W(x,y)$ only for $|x-y|<1$.
Indeed, according to our general results below (see Theorems \ref{Thm2.1} and \ref{thm:cpt}), the assertion of Theorem \ref{T:main} still holds true for the  truncated version of $(\E,\F)$, i.e., the Dirichlet form $(\E^*,\F^*)$ on $L^2(\R^d;\d x)$ given by
$$\E^*(f,g):=\iint_{\{|x-y|<1\}} (f(x)-f(y))(g(x)-g(y))
\frac{W(x,y)}{|x-y|^{d+\alpha}}\,\d x\,\d y$$
and $\F^*:=\overline{C_c^{\infty}(\R^d)}^{\|\cdot\|_{\E^*_1}},$ where $\|f\|_{\E^*_1}=\big(\E^*(f,f)+\|f\|_{L^2(\R^d;\d x)}^2\big)^{1/2}$.
In particular, for the weighted function $W(x,y)$ given by \eqref{e:DF--}, the semigroup associated with $(\E^*,\F^*)$ is compact, if and only if,
$p> 2.$
This shows that in this case the criteria for compactness of the semigroup associated with the non-local Dirichlet form $(\E^*,\F^*)$ on $L^2(\R^d;\d x)$ with finite range jumping kernel is
the
same as these for the local Dirichlet form given by
\eqref{eq:elliptic}.

\item[(ii)] The proof of Theorem \ref{T:main} is based on establishing
the
essential super Poincar\'e inequality
for $(\E,\F)$, which was first introduced by F.-Y. Wang in \cite{W02}.
However, there are few results concerning
the (optimal) essential super Poincar\'e inequality for
non-local Dirichlet forms.
In order to prove Theorem \ref{T:main},
we will make use of the comparison argument and Hardy-type inequality
for
$(\E,\F)$
(see Lemma \ref{lem:hardy-type}).
The test function $\phi$ involved in that inequality is given by
$\phi(x)=(1+|x|^2)^{-\delta/2}$ with $\delta\in (0,1)$.
Then for any $C>0$, there exist $\lambda>0$ and $R_0>0$
such that  for any $x\in \R^d$ with $|x|\geq R_0$,
$$\frac{-{\cal L}_\lambda \phi}{\phi}(x)\geq C,$$
where $\{{\cal L}_\lambda\}_{\lambda\ge1}$ is a family of the
formal
generators
of auxiliary quadratic forms dominated by $(\E, \F)$;
see Propositions \ref{prop:pair} and \ref{prop:ratio}.
This approach is powerful in the sense that
we can even deduce
the corresponding statements for more general non-local Dirichlet forms
with unbounded coefficients.

\item[(iii)]
Recently, the compactness of the weighted fractional heat semigroups
has been studied.
Let $\check{{\cal L}}=-(1+|x|)^p(-\Delta)^{\alpha/2}$ be
a (formal) self-adjoint operator on $L^2(\R^d;\d x/(1+|x|)^p)$ with $p\geq 0$
and $d>\alpha$.
Then, the associated Markovian semigroup is compact,
if and only if $p>\alpha$; see
\cite[Proposition 4.1]{TTT17} and
\cite[Corollary 2.3]{W19}.
See also \cite[Theorems 4.8 and 2.2]{M20+}
for the study of compactness for the semigroup
on
$L^1(\R^d; \d x/(1+|x|)^p)$.
In fact, with the aid of the Hardy inequality or the fractional Sobolev inequality,
the second named author (\cite{W19}) proved that
the following essential super Poincar\'e inequality
\begin{equation}\label{e:comme}
\int_{\R^d} f(x)^2\frac{1}{(1+|x|)^p}\,\d x
\le r \E^0(f,f)+\beta(r)\left(\int_{\R^d} |f(x)| \frac{\psi(x)}{(1+|x|)^p}\,\d x\right)^2
\end{equation}
holds for all $ r>0$ and $f\in C_c^\infty(\R^d)$ with some deceasing function $\beta:(0,\infty)\to (0,\infty)$
and some function $\psi\in L^2(\R^d;\d x)$, if and only if, $p>\alpha$;
see the proof of \cite[Theorem 1.2]{W19} for more details.
Thus, one may try to replace the function $f(x)$ by $f(x)(1+|x|)^{p/2}$ in \eqref{e:comme},
and then to establish the desired essential super Poincar\'e inequality for the Dirichlet form $(\E,\F)$ given by \eqref{e:DF}.
However, it seems that this approach does not work,
and is far from getting  
the sharp result as shown in Theorem \ref{T:main}.
\end{itemize}

The rest of the paper is arranged as follows.
In Section \ref{e:section2}, we present assertions and their proofs
for the compactness and non-compactness criteria
of semigroups associated with general non-local Dirichlet forms with unbounded coefficients; see Theorems \ref{Thm2.1} and \ref{thm:cpt}. As mentioned above, main tasks of the proofs are to disprove and to establish essential super Poincar\'e inequalities for the associated Dirichlet form. In Section \ref{section3},
we present the proof of Theorem \ref{T:main},
and also give 
two examples of non-local Dirichlet forms
with degenerate or singular jumping kernel to illustrate our general results.

\section{General results}\label{e:section2}
Let $J(x,\d y)$ be a nonnegative  kernel on $\R^d\times {\cal B}(\R^d)$
such that the measure $J(\d x, \d y):=J(x,\d y)\,\d x$ satisfies
the symmetry condition
\begin{equation}\label{eq:symmetry}
J(\d x, \d y)= J(\d y, \d x),
\end{equation}
and
\begin{equation}\label{e:core0}
\sup_{x\in \R^d}\int_{\R^d}(1\wedge |x-y|^2)\,J(x,\d y)<\infty.
\end{equation}
That is, $J(x,\d y)$ is regarded as
the jumping kernel
for a pure-jump L\'evy-type process
with bounded coefficients.
Let $W(x,y)$ be a nonnegative, locally bounded and symmetric Borel measurable function
on $\R^d\times \R^d$ such that
\begin{equation}\label{e:core}
x\mapsto \int_{\R^d} (1\wedge |x-y|^2)W(x,y)\,J(x,\d y)\in L^1_{{\rm loc}}(\R^d;\d x).
\end{equation}
Roughly speaking, $W(x,y)$ is regarded as a weighted function.

Let $(\E,\D(\E))$ be a quadratic form on $L^2(\R^d;\d x)$ defined by
\begin{align*}
\D(\E)&=\left\{u\in L^2(\R^d;\d x) :
\iint_{\R^d\times \R^d}
(u(x)-u(y))^2W(x,y)\,J(x,\d y)\,\d x<\infty\right\},\\
\E(u,u)&= \iint_{\R^d\times \R^d}(u(x)-u(y))^2W(x,y)\,J(x,\d y)\,\d x.
\end{align*}
It is easy to verify that
$C_c^{\infty}(\R^d)\subset \D(\E)$
under \eqref{e:core}.
Define the norm $\|\cdot\|_{\E_1}$ on $\D(\E)$ by
$$
\|f\|_{\E_1}=\left(\E(f,f)+\|f\|_{L^2(\R^d;\d x)}^2\right)^{1/2}.
$$
We can then define
$$
\F:=\overline{C_c^{\infty}(\R^d)}^{\|\cdot\|_{\E_1}},
$$
so that $(\E,\F)$ is a regular Dirichlet form on $L^2(\R^d;\d x)$,
see, e.g., \cite[Example 1.2.4]{FOT11}.
We denote by $(P_t)_{t\geq 0}$ the $L^2$-semigroup
associated with $(\E,\F)$.

\subsection{Condition for non-compactness}
In this subsection, we provide a sufficient condition
for non-compactness of $(P_t)_{t\geq 0}$.
\begin{thm}\label{Thm2.1}
Suppose that
\begin{equation}\label{e:sss}\liminf_{l\to\infty} \left[l^{-d}\int_{\{|x|\le l\}}  \int_{\R^d} \left(1\wedge \frac{|x-y|^2}{l^2}\right){W(x,y)}\,J(x,\d y)\,\d x\right]<\infty.
\end{equation}
Then, the semigroup $(P_t)_{t\ge0}$ is not compact on $L^2(\R^d;\d x)$.\end{thm}

\begin{proof}
We apply some
idea in the proof of \cite[Theorem 2.2(ii)]{W19}.
Suppose that $(P_t)_{t\geq 0}$ is compact on $L^2(\R^d;\d x)$.
We then see by \cite[Theorem 0.3.9 and Theorem 3.2.1]{W05} that
for any strictly positive function $\psi\in L^2(\R^d; \d x)$,
there exists a strictly positive nondecreasing function $\beta(r)$ on $(0,\infty)$
such that for any $r>0$ and $f\in C_c^\infty(\R^d)$,
\begin{equation}\label{eq:s-p}
\int_{\R^d}f(x)^2\,\d x\leq r\E(f,f)+\beta(r)\left(\int_{\R^d}
|f(x)|
\psi(x) \,\d x\right)^2.
\end{equation}

We now take $\psi(x)=e^{-|x|}$.
For $l\geq 1$, let $f_l\in C_c^{\infty}(\R^d)$ satisfy
$$f_l(x)\begin{cases}=1, &  0\leq |x|\leq l, \\
\in [0,1],&l\leq |x|\leq 2l,\\
=0, &  |x|\geq 2l
\end{cases}$$ and $\|\nabla f_l\|_\infty\le 2l^{-1}.$
Then
\begin{equation}\label{eq:f_l-1}
\int_{\R^d}f_l(x)^2\,\d x\geq c_1l^d
\end{equation}
and
\begin{equation}\label{eq:f_l-2}
\int_{\R^d}|f_l(x)|\psi(x)\,\d x
\leq \int_{\R^d}\psi(x)\,\d x=c_2.
\end{equation}
Here and in what follows, all the constants $c_i$ are independent of $l$.
Furthermore,
\begin{align*}
\E(f_l,f_l)
 =&\iint_{\R^d\times\R^d}(f_l(x)-f_l(y))^2W(x,y)\,J(x,\d y)\,\d x\\
 =& 2 \int_{\{|x|\leq 2l\}} \int_{\R^d}(f_l(x)-f_l(y))^2W(x,y)\,J(x,\d y)\,\d x\\
 =
 &2\int_{\{|x|\leq 2l\}} \int_{\{|y|<4l\}}(f_l(x)-f_l(y))^2W(x,y)\,J(x,\d y)\,\d x\\
&+2\int_{\{|x|\leq 2l\}} \int_{\{|y|\geq 4l\}}(f_l(x)-f_l(y))^2W(x,y)\,J(x,\d y)\,\d x \\
=&:2{\rm (I)}+2{\rm (II)},
\end{align*} where in the second equality
we used the fact that ${\rm supp}(f_l)\subset B(0,2l)$,
\eqref{eq:symmetry} and the symmetry of $W(x,y)$.
Here and in what follows, $B(0, r):=\{x\in \R^d : |x|<r\}$ for
$r>0$.

We have
\begin{align*}
{\rm (I)} &\leq
 \int_{\{|x|\leq 2l\}} \int_{\{|x-y|<6l\}}(f_l(x)-f_l(y))^2W(x,y)\,J(x,\d y)\,\d x\\
&\le 4l^{-2}
 \int_{\{|x|\leq 2l\}} \int_{\{|x-y|<6l\}}|x-y|^2W(x,y)\,J(x,\d y)\,\d x
\end{align*}
and
\begin{align*}
{\rm (II)}
&\leq \int_{\{|x|\leq 2l\}} \int_{\{|x-y|\geq 2l\}}W(x,y)\,J(x,\d y)\,\d x\\
&\le 4^{-1}l^{-2} \int_{\{|x|\leq 2l\}}
 \int_{\{2l\le |x-y|<
 6l\} }|x-y|^2 W(x,y)\,J(x,\d y)\,\d x\\
&\quad+ \int_{\{|x|\leq 2l\}} \int_{\{ |x-y|\ge 6l\} }W(x,y)\,J(x,\d y)\,\d x.
\end{align*} Hence,
\begin{equation}\label{eq:f_l-3}\E(f_l,f_l)\le c_3 \int_{\{|x|\le 6l\}}
\int_{\R^d} \left(1\wedge \frac{|x-y|^2}{(6l)^2}\right){W(x,y)}\,J(x,\d y)\,\d x.
\end{equation}

Putting \eqref{eq:f_l-1}, \eqref{eq:f_l-2} and \eqref{eq:f_l-3} into \eqref{eq:s-p}
with $f=f_l$,
we find that there are constants $c_4,c_5>0$ such that for all $l\ge1$ and $r>0$.
$$l^d\le c_4 r \int_{\{|x|\le 6l\}}  \int_{\R^d} \left(1\wedge \frac{|x-y|^2}{(6l)^2}\right){W(x,y)}\,J(x,\d y)\,\d x+ c_5\beta(r);$$
that is,
$$
1\le c_4 r\left[l^{-d}\int_{\{|x|\le 6l\}} \int_{\R^d}
\left(1\wedge \frac{|x-y|^2}{(6l)^2}\right){W(x,y)}\,J(x,\d y)\,\d x\right]
+ c_5\beta(r)l^{-d}.$$
However, under \eqref{e:sss}, we have a contradiction from the inequality above
by taking liminf for  $l\rightarrow\infty$ and then
letting
$r\to 0$.
Namely, $(P_t)_{t\geq 0}$ is not compact on $L^2(\R^d;\d x)$.
\end{proof}

\begin{rem}\label{rem:diff}\rm
Let $(\E^L,C_c^{\infty}(\R^d))$ be a quadratic form as in \eqref{eq:elliptic}.
Suppose that $\{a_{ij}(x)\}_{1\leq i,j\leq d}$ is local uniformly elliptic,
and that
there exists a positive Borel measurable function $a(x)$ on $\R^d$
such that for any $x,\xi\in \R^d$,
$$
\sum_{i,j=1}^da_{ij}(x)\xi_i\xi_j\leq a(x)|\xi|^2.
$$
Denote by $\F^L$ the closure of $C_c^{\infty}(\R^d)$, i.e., $\F^L:=\overline{C_c^{\infty}(\R^d)}^{\|\cdot\|_{\E^L_1}}$ with
$\|f\|_{\E^L_1}=\big(\E^L(f,f)+\|f\|_{L^2(\R^d;\d x)}^2\big)^{1/2}.$
By analogy with  the proof of Theorem \ref{Thm2.1},
we can show that
the semigroup associated with
$(\E^L,\F^L)$ is non-compact, if
$$
\liminf_{l\rightarrow\infty}\left[l^{-(d+2)}\int_{\{|x|\leq 2l\}}a(x)\,\d x\right]<\infty.
$$
In particular, if $\{a_{ij}(x)\}_{1\leq i,j\leq d}$ satisfies \eqref{eq:coeff}, then
 the semigroup associated with $(\E^L,\F^L)$ is non-compact when $p\leq 2$.
\end{rem}

\subsection{Condition for compactness}
We first recall that $W(x,y)$ and $J(x,\d y)$ satisfy \eqref{e:core}.
To establish a sufficient condition
for
compactness of the semigroup $(P_t)_{t\ge0}$,
we will impose the following conditions on $W(x,y)$ and $J(x,\d y)$, respectively.
\begin{assum}\label{assum:w}\rm
\begin{enumerate}
\item[(i)]  The function $W(x,y)$ is locally bounded on $\R^d\times \R^d$ such that
for  any $\lambda>0$,
there exist a strictly positive $C^1$-function $V_{\lambda}(x)$ on $\R^d$
and a constant $R_0>0$ so that
\begin{itemize}
\item $W(x,y)\geq V_{\lambda}(x)+V_{\lambda}(y)$
for any $x,y\in \R^d$ with $|x-y|<1$;
\item $\inf_{x\in \R^d}V_\lambda(x)>0$, and $V_{\lambda}(x)=\lambda(1+|x|^2)$ for any $x\in B(0, R_0)^c$.
\end{itemize}
\item[(ii)]  There exists a nonnegative Borel measure $\nu(\d z)$
on $[0,1)$ such that
\begin{itemize}
\item $J(x,x+A)\ge \nu(A)$ for any $x\in \R^d$ and Borel set $A\subset B(0,1)$;
\item $\nu(A)=\nu(-A)$ for any Borel set $A\subset B(0,1)$,
where $-A=\{x\in \R^d \mid -x\in A\}$;
\item Let
$$\varphi(\xi)=\int_{\{|z|< 1\}} (1-\cos \langle z,\xi\rangle)\,\nu(\d z),\quad \xi\in \R^d$$ and
$$
\beta_0(r)=\int_{\R^d} e^{-r\varphi(\xi)}\,\d \xi,\quad r>0.$$
Then
$$\beta_0(r)<\infty,\quad \int_r^\infty \frac{\beta_0^{-1}(s)}{s}\,\d s<\infty,\quad r>0,$$
where
$\beta_0^{-1}(r)$ is an inverse function of $\beta_0(r)$.
\end{itemize}
\end{enumerate}
\end{assum}

Under Assumption \ref{assum:w}(i), we have
$$\inf_{x,y\in \R^d: |x-y|<1}W(x,y)>0.$$
Note also that
by \eqref{e:core0},
the measure $\nu$ in Assumption \ref{assum:w}(ii) satisfies
\begin{equation}\label{e:moment}
\int_{\{0<|z|<1\}}|z|^2\,\nu(\d z)
\le \sup_{x\in \R^d}
\int_{\{0<|x-y|<1\}}|x-y|^2J(x,y)\,\d y<\infty.
\end{equation}
We then have
\begin{thm}\label{thm:cpt}
Under Assumption {\rm \ref{assum:w}},
the semigroup $(P_t)_{t\geq 0}$ is compact on $L^2(\R^d;\d x)$.
\end{thm}

To prove Theorem \ref{thm:cpt},
we need to introduce a class of auxiliary quadratic forms. For any fixed $\lambda>0$,
we define $W_{\lambda}(x,y)=V_{\lambda}(x)+V_{\lambda}(y)$, and
$$
\E^{\lambda}(u,u)
=\iint_{\{0<|z|<1\}}(u(x+z)-u(x))^2W_{\lambda}(x,x+z)\,\nu(\d z)\,\d x,\quad  u\in\F.
$$
Then by Assumption \ref{assum:w},
$\E^{\lambda}(u,u)\leq \E(u,u)$ for any $u\in \F$.

For our purpose, we will extend
the domain of the quadratic form $(\E^{\lambda},\F)$.
For a pair $f$ and $g$ of Borel measurable functions on $\R^d$
such that
$$
\iint_{\{0<|z|<1\}}|f(x+z)-f(x)||g(x+z)-g(x)| W_{\lambda}(x,x+z)\,\nu(\d z)\,\d x<\infty,
$$
we let
$$
\E^{\lambda}(f,g):=\iint_{\{0<|z|<1\}}(f(x+z)-f(x))(g(x+z)-g(x))
W_{\lambda}(x,x+z)\,\nu(\d z)\,\d x.
$$
It is clear that $\E^{\lambda}(f,g)$ is well defined,
if $f,g\in \F$ (in particular, if $f,g\in C_c^\infty(\R^d)$).

\begin{prop}\label{prop:pair}
Let $\phi(x)=(1+|x|^2)^{-\delta/2}$ for some $\delta>0$.
Under Assumption {\rm \ref{assum:w}}, we have the following statements
for any $\lambda>0$.
\begin{enumerate}
\item[{\rm (i)}] For any $g\in C_c^{\infty}(\R^d)$,
$$
\iint_{\{0<|z|<1\}}|\phi(x+z)-\phi(x)||g(x+z)-g(x)|W_{\lambda}(x,x+z)\,\nu(\d z)\,\d x<\infty.
$$
In particular, $\E^\lambda(\phi, g)$ is well defined.
\item[{\rm (ii)}] For any $g\in C_c^{\infty}(\R^d)$,
$$
\iint_{\{0<|z|<1\}}|\phi(x+z)-\phi(x)-\langle \nabla \phi(x),z\rangle||g(x)|
W_{\lambda}(x,x+z)\,\nu(\d z)\,\d x<\infty
$$
and
$$
\iint_{\{0<|z|<1\}}|\langle \nabla \phi(x),z\rangle||g(x)|
|W_{\lambda}(x,x+z)-W_{\lambda}(x,x-z)|\,\nu(\d z)\,\d x<\infty.
$$
\item[{\rm (iii)}]
For any $g\in C_c^{\infty}(\R^d)$,
$$
\E^{\lambda}(\phi,g)=-2\int_{\R^d}({\cal L}_{\lambda} \phi)(x)g(x)\, \d x,
$$
where
\begin{equation}\label{eq:generator-0}
\begin{split}
{\cal L}_{\lambda}\phi(x)
&=\int_{\{0<|z|<1\}}(\phi(x+z)-\phi(x)-\langle \nabla \phi(x), z\rangle)
W_{\lambda}(x,x+z)\,\nu(\d z)\\
&\quad+\frac{1}{2}\int_{\{0<|z|<1\}}\langle \nabla \phi(x), z\rangle
(W_{\lambda}(x,x+z)-W_{\lambda}(x,x-z))\,\nu(\d z).
\end{split}
\end{equation}
In particular, ${\cal L}_{\lambda}\phi$ is locally bounded on $\R^d$.
\end{enumerate}
\end{prop}

\begin{proof}
We first prove (i).
Let $g\in C_c^{\infty}(\R^d)$.
Since $g$ has compact support,
there exists a compact set $K\subset \R^d$ such that
\begin{align*}
&\iint_{\{0<|z|<1\}}|\phi(x+z)-\phi(x)||g(x+z)-g(x)|W_{\lambda}(x,x+z)\,\nu(\d z)\,\d x\\
&=\int_K \int_{\{0<|z|<1\}}
|\phi(x+z)-\phi(x)||g(x+z)-g(x)|W_{\lambda}(x,x+z)\,\nu(\d z) \,\d x.
\end{align*}
By the mean value theorem,
there exists $c_1>0$ such that for any $x,z\in \R^d$ with $|z|<1$,
$$
|\phi(x+z)-\phi(x)|\leq \frac{c_1|z|}{(1+|x|^2)^{(\delta+1)/2}}\leq c_1|z|.
$$
Then, according to \eqref{e:moment}
and the local boundedness of $W_{\lambda}(x,y)$,
\begin{align*}
&\int_K \int_{\{0<|z|<1\}}(\phi(x+z)-\phi(x))^2
W_{\lambda}(x,x+z)\,\nu(\d z) \,\d x\\
&\leq c_2 \int_K \int_{\{0<|z|<1\}}|z|^2W_{\lambda}(x,x+z)\,\nu(\d z)
 \,\d x
 <\infty.
\end{align*}
Hence by the Cauchy-Schwarz inequality,
\begin{align*}
&\left[\int_K \int_{\{0<|z|<1\}}
|\phi(x+z)-\phi(x)||g(x+z)-g(x)|W_{\lambda}(x,x+z)\,\nu(\d z) \,\d x\right]^2\\
&\leq \int_K \int_{\{0<|z|<1\}}
(\phi(x+z)-\phi(x))^2W_{\lambda}(x,x+z)\,\nu(\d z) \,\d x\\
&\quad\times
\int_K \int_{\{0<|z|<1\}}
(g(x+z)-g(x))^2W_{\lambda}(x,x+z)\,\nu(\d z) \,\d x<\infty,
\end{align*}
which implies  (i).

We next prove (ii) and (iii) in a similar way to the proof of \cite[Theorem 2.2]{SU14}.
Fix $\varepsilon>0$.
Then, by the symmetry of $\nu(\d z)$ and $W_{\lambda}(x,y)$,
\begin{equation}\label{eq:ibp}
\begin{split}
&\iint_{\{\varepsilon<|z|<1\}}(\phi(x+z)-\phi(x))(g(x+z)-g(x))
W_{\lambda}(x,x+z)\,\nu(\d z)\,\d x\\
&=-2\iint_{\{\varepsilon<|z|<1\}}(\phi(x+z)-\phi(x))g(x)
W_{\lambda}(x,x+z)\,\nu(\d z)\,\d x\\
&=-2\iint_{\{\varepsilon<|z|<1\}}
(\phi(x+z)-\phi(x)-\langle \nabla \phi(x),z\rangle)g(x)
W_{\lambda}(x,x+z)\,\nu(\d z)\,\d x\\
&\quad-2\iint_{\{\varepsilon<|z|<1\}}\langle \nabla \phi(x),z\rangle g(x)
W_{\lambda}(x,x+z)\,\nu(\d z)\,\d x.
\end{split}
\end{equation}

By the Taylor theorem,
 there exists $c_3>0$ such that for any $x,z\in \R^d$ with $|z|<1$,
$$
|\phi(x+z)-\phi(x)-\langle\nabla \phi(x), z\rangle|
\leq \frac{c_3|z|^2}{(1+|x|^2)^{(\delta+2)/2}}\leq c_3|z|^2,
$$
which implies that
\begin{align*}
&\int_{\{\varepsilon<|z|<1\}}
|\phi(x+z)-\phi(x)-\langle \nabla \phi(x),z\rangle|
W_{\lambda}(x,x+z)\,\nu(\d z)\\
&\leq c_3
\int_{\{0<|z|<1\}}|z|^2W_{\lambda}(x,x+z)\,\nu(\d z)=:M_1(x).
\end{align*}
Note that, by
\eqref{e:moment}
and the local boundedness of $W_\lambda(x,y)$ again,
the function $M_1(x)$ is
locally bounded on $\R^d$.
Since $g$ has compact support in $\R^d$,
we have
\begin{equation}\label{eq:abs-conv}
\begin{split}
&\iint_{\{\varepsilon<|z|<1\}}|\phi(x+z)-\phi(x)-\langle \nabla \phi(x),z\rangle||g(x)|
W_{\lambda}(x,x+z)\,\nu(\d z)\,\d x\\
&\leq \int_{\R^d}M_1(x)|g(x)|\,\d x
<\infty.
\end{split}
\end{equation}

Since the measure $\nu(\d z)$ is symmetric, we have
\begin{equation}\label{e:gene-drift}
\begin{split}
&\iint_{\{\varepsilon<|z|<1\}}\langle \nabla \phi(x),z\rangle
g(x)W_{\lambda}(x,x+z)\,\nu(\d z)\,\d x\\
&=\frac{1}{2}\iint_{\{\varepsilon<|z|<1\}}
\langle \nabla \phi(x),z\rangle g(x)
(W_{\lambda}(x,x+z)-W_{\lambda}(x,x-z))\,\nu(\d z)\,\d x.
\end{split}
\end{equation}
By the mean value theorem and $V_\lambda\in C^1(\R^d)$,
there exists a locally bounded nonnegative function
$c_{\lambda}(x,z)$ on $\R^d\times\R^d$
such that for any $x,z\in \R^d$ with $|z|<1$,
$$
|W_{\lambda}(x,x+z)-W_{\lambda}(x,x-z)|
=|V_{\lambda}(x+z)-V_{\lambda}(x-z)|
\leq c_{\lambda}(x,z)|z|,
$$
which implies that
\begin{align*}
&\int_{\{\varepsilon<|z|<1\}}
|\langle \nabla \phi(x),z\rangle|
|W_{\lambda}(x,x+z)-W_{\lambda}(x,x-z)|\,\nu(\d z)\\
&\leq c_4
\int_{\{0<|z|<1\}}|z|^2
c_{\lambda}(x,z)\,\nu(\d z)=:M_2(x).
\end{align*}
Then, thanks to \eqref{e:moment}
again,
the function $M_2(x)$ is also locally bounded on $\R^d$.
In particular,
\begin{equation}\label{e:abs-conv-2}
\begin{split} 
&\iint_{\{\varepsilon<|z|<1\}}
|\langle \nabla \phi(x),z\rangle| |g(x)|
|W_{\lambda}(x,x+z)-W_{\lambda}(x,x-z)|\,\nu(\d z)\,\d x\\
&\leq \int_{\R^d}M_2(x)|g(x)|\,\d x<\infty.
\end{split}
\end{equation} 
We thus have the assertion (ii) by letting $\varepsilon\rightarrow 0$
in \eqref{eq:abs-conv} and \eqref{e:abs-conv-2}.
We also obtain the assertion (iii) by
 \eqref{eq:ibp}, \eqref{e:gene-drift} and (ii).
\end{proof}

\begin{prop}\label{prop:ratio}
For any $\delta\in (0,1)$ and $C>0$,
there exist $\lambda>0$ and $R_0>0$
such that for any $x\in B(0,R_0)^c$,
the function $\phi(x)=(1+|x|^2)^{-\delta/2}$ satisfies
\begin{equation}\label{eq:ratio}
\frac{-{\cal L}_{\lambda}\phi}{\phi}(x)\geq C.
\end{equation}
In particular,
the function $-{\cal L}_{\lambda}\phi/\phi$
is bounded from below.
\end{prop}

\begin{proof}
According to the proof of Proposition \ref{prop:pair},
we know that ${\cal L}_{\lambda}\phi$
is locally bounded.
So it suffices to prove \eqref{eq:ratio}.

Let $\delta>0$. Then for any $x,z\in \R^d$ with $0<|z|<1$,
we have by the Taylor theorem,
$$
\phi(x+z)-\phi(x)-\langle \nabla \phi(x), z\rangle
=\delta\int_0^1 (1-s)
\frac{(\delta+2)\langle x+sz,z\rangle^2-(1+|x+sz|^2)|z|^2}
{(1+|x+sz|^2)^{\delta/2+2}}\,\d s.
$$
Then, the Fubuni theorem yields that for any $\lambda>0$,
\begin{align}\label{e:bound-1}
&\int_{\{0<|z|<1\}}(\phi(x+z)-\phi(x)-\langle \nabla \phi(x), z\rangle)
W_{\lambda}(x,x+z)\,\nu(\d z)\nonumber\\
&=\delta\int_0^1(1-s)
 \int_{\{0<|z|<1\}}
\frac{(\delta+2)\langle x+sz,z\rangle^2-(1+|x+sz|^2)|z|^2}{(1+|x+sz|^2)^{\delta/2+2}}
W_{\lambda}(x,x+z)\,\nu(\d z) \,\d s\nonumber\\
&=\delta(\delta+2)\int_0^1(1-s)
 \int_{\{0<|z|<1\}}\frac{\langle x+sz,z\rangle^2}{(1+|x+sz|^2)^{\delta/2+2}}
W_{\lambda}(x,x+z)\,\nu(\d z) \,\d s\nonumber\\
&\quad -\delta\int_0^1(1-s) \int_{\{0<|z|<1\}}\frac{|z|^2}{(1+|x+sz|^2)^{\delta/2+1}}
W_{\lambda}(x,x+z)\,\nu(\d z) \,\d s\nonumber\\
&=:{\rm (I)}-{\rm (II)}.
\end{align}

For any $\delta>0$, $\varepsilon\in (0,1)$ and $\lambda>0$,
there exists a constant $R_0>0$ by Assumption \ref{assum:w}(i)
such that
for any $x\in B(0,R_0)^c$ and $z\in \R^d$ with $ 0<|z|<1 $,
\begin{equation}\label{eq:comparison1}
W_\lambda(x,x\pm z)=\lambda(1+|x|^2)+\lambda(1+|x\pm z|^2){\rm ;}
\end{equation}
in particular, for any $x\in B(0,R_0)^c$, $z\in \R^d$ with $0<|z|<1$, and $s\in (0,1)$,
\begin{equation}\label{eq:comparison-u}
\frac{W_{\lambda}(x,x+z)}{(1+|x+sz|^2)^{\delta/2+2}}
\le \ \frac{2\lambda(1+\varepsilon)}{(1+|x|^2)^{\delta/2+1}}
\end{equation}
and
\begin{equation}\label{eq:comparison-l}
\frac{W_{\lambda}(x,x+z)}{(1+|x+sz|^2)^{\delta/2+1}}
\ge\frac{2\lambda(1-\varepsilon)}{(1+|x|^2)^{\delta/2}}
=2\lambda(1-\varepsilon)\phi(x).
\end{equation}
Then, by \eqref{eq:comparison-u},
\begin{equation}\label{eq:i-bound-1}
{\rm (I)}
\leq \delta(\delta+2)\frac{2\lambda(1+\varepsilon) 
}{(1+|x|^2)^{\delta/2+1}}
\int_0^1(1-s) \left(\int_{\{0<|z|<1\}}
\langle x+sz,z\rangle^2\,\nu(\d z)\right) \,\d s\\
=:{\rm (I)_*}.
\end{equation}

By the symmetry of the measure $\nu(\d z)$,
$$
\int_{\{0<|z|<1\}}\langle x,z\rangle |z|^2\,\nu(\d z)=0.
$$
Since
$$
\langle x+sz,z\rangle^2
=(\langle x, z\rangle+s|z|^2)^2
=\langle x,z\rangle^2+2s\langle x,z\rangle |z|^2
+s^2|z|^4,
$$
we have
$$
\int_{\{0<|z|<1\}}\langle x+sz,z\rangle^2\,\nu(\d z)
=F(x)|x|^2+c_1s^2,
$$
where
$$
F(x)=\frac{1}{|x|^2}\int_{\{0<|z|<1\}}\langle x,z \rangle^2\,\nu(\d z),
\quad
c_1=\int_{\{0<|z|<1\}}|z|^4\,\nu(\d z).
$$
Therefore, for any $s\in (0,1)$ and $x\in \R^d$,
\begin{align*}
\int_0^1(1-s)
\left(\int_{\{0<|z|<1\}}\langle x+sz,z\rangle^2\,\nu(\d z)\right) \,\d s
=\frac{F(x)}{2}|x|^2+\frac{c_1}{12}.
\end{align*}
Then, by \eqref{eq:i-bound-1}, we obtain for any $x\in B_0(0,R_0)^c$,
\begin{equation}\label{eq:i-bound-2}
\begin{split}
{\rm (I)}\leq {\rm (I)_*}
&\leq \delta(\delta+2)
\frac{2\lambda(1+\varepsilon)
}{(1+|x|^2)^{\delta/2+1}}
\left(\frac{F(x)}{2}|x|^2+\frac{c_1}{12}\right)\\
&\leq  \lambda (1+\varepsilon)\delta(\delta+2)
\left(F(x)+\frac{c_1}{6|x|^2}\right)\phi(x).
\end{split}
\end{equation}

Let
$$
\gamma_0=\int_{\{0<|z|<1\}}|z|^2\,\nu(\d z).
$$
Then, by \eqref{eq:comparison-l},
we have for any $x\in B(0,R_0)^c$,
\begin{equation}\label{e:bound-2}
{\rm (II)}
\ge 2\lambda\delta(1-\varepsilon)\phi(x)
\int_0^1(1-s)  \,\d s\left(\int_{\{0<|z|<1\}}|z|^2\,\nu(\d z)\right)
=\lambda\delta\gamma_0(1-\varepsilon)\phi(x).
\end{equation}
Combining this with \eqref{eq:i-bound-2},
we get that for any $x\in B(0,R_0)^c$,
\begin{equation}\label{eq:upper}
\begin{split}
&\int_{\{0<|z|<1\}}(\phi(x+z)-\phi(x)-\langle \nabla \phi(x), z\rangle)
W_{\lambda}(x,x+z)\,\nu(\d z)\\
&={\rm (I)}-{\rm (II)}\\
&\leq \lambda \delta
\left\{(1+\varepsilon)(\delta+2)\left(F(x)+\frac{c_1}{6|x|^2}\right)-(1-\varepsilon)\gamma_0\right\}\phi(x).
\end{split}
\end{equation}

Since, by \eqref{eq:comparison1},
\begin{align*}
W_{\lambda}(x,x+z)-W_{\lambda}(x,x-z)
=\lambda(1+|x+z|^2)-\lambda(1+|x-z|^2)
=4\lambda\langle x,z\rangle,
\end{align*}
we have
\begin{align*}
\int_{\{0<|z|<1\}}\langle x, z\rangle
(W_{\lambda}(x,x+z)-W_\lambda(x,x-z))\,\nu(\d z)
&=4\lambda\int_{\{0<|z|<1\}}\langle x, z\rangle^2 \,\nu(\d z)\\
&= 4\lambda F(x)|x|^2,
\end{align*}
which implies that  for any $x\in B(0,R_0)^c$,
\begin{equation}\label{e:drift-bound}
\begin{split}
&\frac{1}{2}\int_{\{0<|z|<1\}}\langle\nabla \phi(x),z\rangle
(W_{\lambda}(x,x+z)-W_{\lambda}(x,x-z))\,\nu(\d z)\\
&=-\frac{\delta}{2(1+|x|^2)^{\delta/2+1}}
\int_{\{0<|z|<1\}}\langle x, z\rangle
(W_{\lambda}(x,x+z)-W_{\lambda}(x,x-z))\,\nu(\d z)\\
&\leq -2\lambda \delta (1-\varepsilon)F(x)\phi(x).
\end{split}
\end{equation}
The last inequality above holds true
due to the fact that we can choose $R_0$ large enough if necessary.
Putting \eqref{e:drift-bound}
and \eqref{eq:upper} into \eqref{eq:generator-0},
we get for any $x\in B(0,R_0)^c$
\begin{equation}\label{e:bound}
\begin{split}
{\cal L}_{\lambda}\phi(x)
&\leq \lambda \delta
\left\{(1+\varepsilon)(\delta+2)\left(F(x)+\frac{c_1}{6|x|^2}\right)
-(1-\varepsilon)(\gamma_0+2F(x))\right\}\phi(x)\\
&=\lambda \delta
\left\{((1+\varepsilon)\delta+4\varepsilon)F(x)+\frac{c_1(1+\varepsilon)(\delta+2)}{6|x|^2}
-(1-\varepsilon)\gamma_0\right\}\phi(x)\\
&\leq \lambda \delta
\left\{((1+\varepsilon)\delta+5\varepsilon-1)\gamma_0+\frac{c_1(1+\varepsilon)(\delta+2)}{6|x|^2}\right\}\phi(x)\\
&\leq \lambda \delta\gamma_0
\left\{(1+\varepsilon)\delta+6\varepsilon-1\right\}\phi(x).
\end{split}
\end{equation}
Here
in the second inequality
we used the fact that $F(x)\leq \gamma_0$ for any $x\in B(0,R_0)^c$,
and the last inequality holds true again
because
we can choose $R_0$ large enough if necessary.

For any $\delta\in (0,1)$,
we can take $\varepsilon\in (0,1)$
so small that
$$
(1+\varepsilon)\delta+6\varepsilon-1<0.
$$
Then, \eqref{eq:ratio} follows  by letting in \eqref{e:bound}
$$
\lambda
=\frac{C}{\delta\gamma_0\{(1-6\varepsilon)-(1+\varepsilon)\delta\}}.
$$
Hence the proof is complete.
\end{proof}

\begin{rem}\label{e:remark-key}\rm
As seen from \eqref{e:bound-1}, \eqref{e:bound-2} and \eqref{e:bound} above, the negativity of ${\cal L}_{\lambda}\phi$ is mainly due to
the estimate for the term ${\rm (II)}$.
Furthermore,  ${\rm (II)}$ is deduced from the Taylor formula for the test function $\phi$, and roughly speaking it comes from the expansion term
$\frac{\partial ^2 \phi 
}{\partial x_i\partial x_j}$. 
This indicates that, for the generator 
${\cal L}_\lambda$ of the quadratic form $(\E^{\lambda},\F)$ with finite range jumping kernel, the estimate of ${\cal L}_\lambda\phi(x)$ for $|x|$ large enough is similar to the second order elliptic operator acting on $\phi$.
We also note that,
the condition $\delta\in (0,1)$ and the
inequality $F(x)\leq \gamma_0$ are crucial for the proof of Proposition \ref{prop:ratio}.
\end{rem}

We also need the following Hardy type inequality
for the Dirichlet form $(\E,\F)$.
\begin{lem}\label{lem:hardy-type}
Let $\lambda>0$ and $\phi(x)=(1+|x|^2)^{-\delta/2}$ for some $\delta>0$.
If the function $-{\cal L}_{\lambda}\phi/\phi$ is bounded from below,
then for any $f\in \F$,
$$\E(f,f)\geq 2\int_{\R^d}\frac{-{\cal L}_{\lambda}\phi}{\phi}f^2\,\d x.$$
\end{lem}

\begin{proof}
We first assume that $f\in C_c^{\infty}(\R^d)$.
Since $f^2/\phi\in C_c^{\infty}(\R^d)$,
we have  by Proposition \ref{prop:pair}(iii),
\begin{equation}\label{eq:generator}
\begin{split}
&\int_{\R^d}\frac{-{\cal L}_{\lambda}\phi}{\phi}(x)f(x)^2\,\d x\\
&
=\int_{\R^d}(-{\cal L}_{\lambda}\phi)(x)\frac{f^2}{\phi}(x)\,\d x=\frac{1}{2}\E^{\lambda}(\phi,f^2/\phi)\\
&=\frac{1}{2}\iint_{\{0<|z|<1\}}(\phi(x+z)-\phi(x))
\left(\frac{f^2}{\phi}(x+z)-\frac{f^2}{\phi}(x)\right)W_{\lambda}(x,x+z)
\,\nu(\d z)\,\d x.
\end{split}
\end{equation}
Because for any $x,y\in \R^d$,
\begin{align*}
(\phi(y)-\phi(x))\left(\frac{f^2}{\phi}(y)-\frac{f^2}{\phi}(x)\right)
&=f(y)^2-\left(f(x)^2\frac{\phi(y)}{\phi(x)}
+f(y)^2\frac{\phi(x)}{\phi(y)}\right)+f(x)^2\\
&\leq f(y)^2-2|f (y)||f (x)|+f(x)^2\leq (f(y)-f(x))^2,
\end{align*}
we have by \eqref{eq:generator} and  Assumption \ref{assum:w},
$$
\int_{\R^d}\frac{-{\cal L}_{\lambda}\phi}{\phi}(x)f(x)^2\,\d x
\leq \frac{1}{2}\iint_{\{0<|z|<1\}}(f(x+z)-f(x))^2
W_{\lambda}(x,x+z)\,\nu(\d z)\,\d x\leq \frac{1}{2}\E(f,f).
$$

We next assume that $f\in {\cal F}$.
Take a sequence $\{f_n\}_{n\ge1}\subset  C_c^{\infty}(\R^d)$
such that $\|f_n-f\|_{\E_1}\rightarrow 0$ as $n\rightarrow\infty$.
Let  $\{f_{n_k}\}_{k\ge1}$ be a subsequence such that $f_{n_k}\rightarrow f$, a.e.
Since the function $-{\cal L}_{\lambda}\phi/\phi$ is bounded from below
by assumption,
we get by an application of Fatou's lemma,
$$\E(f,f)=\lim_{k\rightarrow\infty}\E(f_{n_k},f_{n_k})
\geq 2
\liminf_{k\rightarrow\infty}\int_{\R^d}
\frac{-{\cal L}_{\lambda}\phi}{\phi}(x)f_{n_k}(x)^2\,\d x
\geq 2\int_{\R^d}\frac{-{\cal L}_{\lambda}\phi}{\phi}(x)f(x)^2\,\d x.$$
This completes the proof.
\end{proof}

Now, we are in a position to present the
\begin{proof}[Proof of Theorem {\rm \ref{thm:cpt}}]
We first note that the Nash-type inequality holds for $(\E,\F)$.
In fact, let $Z:=(Z_t)_{t\ge0}$ be a L\'evy process on $\R^d$
with L\'evy measure ${\bf 1}_{\{0<|z|<1\}}\,\nu(\d z)$.
According to \cite[Proposition 4.1]{VS13} and Assumption \ref{assum:w}(ii),
the process $Z$ has a transition density function $p(t,x)$
with respect to the Lebesgue measure
such that for all $t>0$,
$$\|p(t,\cdot)\|_\infty \le
(2\pi)^{-d/2}\int_{\R^d} e^{-t \varphi(\xi)}\,\d \xi\le \beta_0(t)<\infty.$$
This along with \cite[Theorem 3.3.15]{W05} yields that
for any $r>0$ and $f\in C_c^\infty(\R^d)$,
\begin{equation}\label{e:levy}
\int_{\R^d} f(x)^2\,\d x
\le \frac{r}{2} \iint_{ \{0<|z|<1\} } {(f(x+z)-f(x))^2}\nu(\d y)\d x
+\beta_0(r)\left(\int_{\R^d}|f(x)|\,\d x\right)^2.
\end{equation}
According to Assumption \ref{assum:w},
we further know that there is a constant $c_1>0$ such that
for any $r>0$ and $f\in C_c^\infty(\R^d)$,
$$\int_{\R^d} f(x)^2\,\d x
\le c_1r\E(f,f)+\beta_0(r)\left(\int_{\R^d}|f(x)|\,\d x\right)^2.$$
Replacing $r$ with  $r/c_1$ in the inequality above,
we see that  for all $r>0$ and $f\in C_c^\infty(\R^d)$,
$$\int_{\R^d} f(x)^2\,\d x
\le r\E(f,f)+\beta_0(r/c_1)\left(\int_{\R^d}|f(x)|\,\d x\right)^2.$$
In particular,  by Assumption \ref{assum:w}(ii) and \cite[Theorem 3.3.14]{W05},
the associated semigroup $(P_t)_{t\ge0}$ also has a transition density function with respect to the Lebesgue measure.

We next establish the essential super Poincar\'e inequality for $(\E,\F)$.
Let $\phi(x)=(1+|x|^2)^{-\delta/2}$ for some fixed $\delta\in (0,1)$.
According to Proposition \ref{prop:ratio},
for any $C>0$,
we can take positive constants
$\lambda:=\lambda(C)$, $C_0:=C_0(C)$ and $R_0:=R_0(C)$
so that for all $x\in \R^d$,
$$\frac{-{\cal L}_{\lambda}\phi}{\phi}(x)
\ge C{\bf 1}_{B(0,R_0)^c}(x)-C_0{\bf 1}_{B(0,R_0)}(x).$$
This along with Lemma \ref{lem:hardy-type} yields that
for any $f\in C_c^\infty(\R^d)$,
\begin{align*}
C\int_{B(0,R_0)^c}f(x)^2\,\d x
&\le \int_{\R^d}\frac{-{\cal L}_{\lambda}\phi}{\phi}(x)f(x)^2\,\d x
+C_0\int_{B(0,R_0)}f(x)^2\,\d x\\
&\le \frac{1}{2}\E(f,f)+C_0\int_{B(0,R_0)}f(x)^2\,\d x,
\end{align*}
whence
\begin{align}\label{eq:poincare}
\int_{B(0,R_0)^c}f(x)^2\,\d x
\le  \frac{1}{2C}\E(f,f)+\frac{C_0}{C}\int_{B(0,R_0)}f(x)^2\,\d x.
\end{align}

On the other hand,
let $\eta:=\eta_{R_0}\in C_c^\infty(\R^d)$ take values in $[0,1]$
so that $\eta=1$ on $B(0,R_0)$ and $\eta=0$ on $B(0,R_0+1)^c$.
In particular, there exists a constant $c_2>0$, independently of $R_0$, such that
for any $x,z\in \R^d$,
\begin{equation}\label{eq:eta}
|\eta(x+z)-\eta(x)|\leq c_2|z|.
\end{equation}
Then, by \eqref{e:levy},
we have for any $s>0$ and $f\in C_c^\infty(\R^d)$,
\begin{equation}\label{e:bound-0}
\begin{split}
 \int_{B(0,R_0)}f(x)^2\,\d x
&\le  \int_{\R^d} (f(x)\eta(x)) ^2 \,\d x \\
&\le \frac{s}{2}   \iint_{\{ 0<|z|<1 \}} {(f(x+z)\eta(x+z) -f(x)\eta(x))^2}\,\nu(\d z)\,\d x\\
&\quad+\beta_0(s)\left(\int_{\R^d}|f(x)\eta(x)|\,\d x\right)^2\\
&\le  s   \iint_{ \{0<|z|<1 \}} {(f(x+z)-f(x))^2}\,\nu(\d z)\,\d x\\
&\quad+ s\int_{\R^d} f(x)^2
\left(\int_{ \{0<|z|<1 \}} {(\eta(x+z) -\eta(x))^2}\,\nu(\d z)\right)\,\d x \\
&\quad+\beta_0(s)\left(\int_{B(0, R_0+1)}|f(x)|\,\d x\right)^2\\
&\le c_3 s \E (f,f) +c_3 s \int_{\R^d}f(x)^2\,\d x
+\beta_0(s)\left(\int_{B(0, R_0+1)}|f(x)|\,\d x\right)^2,
\end{split}
\end{equation}
where in the last inequality we used
\eqref{eq:eta}
and \eqref{e:moment}
to show that
$$\sup_{x\in \R^d}\int_{ \{0<|z|<1 \}} {(\eta(x+z)-\eta(x))^2}\,\nu(\d z)
\leq c_3\int_{ \{0<|z|<1\} } |z|^2\,\nu(\d z)<\infty.$$
Combining \eqref{e:bound-0}
with \eqref{eq:poincare},
we know that for any $f\in C_c^\infty(\R^d)$ and $s>0$,
\begin{align*}
\int_{\R^d} f(x)^2\,\d x
&\le \left(\frac{1}{2C}+c_3s\right)\E(f,f)+c_3 s \int_{\R^d} f(x)^2\,\d x \\
&\quad +\beta_0(s)\left(\int_{B(0, R_0+1)}|f(x)|\,\d x\right)^2
+\frac{C_0}{C}\int_{B(0,R_0)}f(x)^2\,\d x\\
&\le  \left\{\frac{1}{2C}+c_3s\left(1+\frac{C_0}{C} \right)\right\}\E(f,f)
+c_3s\left(1+\frac{C_0}{C}\right)\int_{\R^d} f(x)^2\,\d x\\
&\quad +\beta_0(s)\left(1+\frac{C_0}{C}\right)
\left(\sup_{z\in B(0,R_0+1)}\psi(z)^{-1}\right)
\left(\int_{\R^d} |f(x)|\psi(x)\,\d x\right)^2,
\end{align*}
where $\psi$ is any  strictly positive function in $L^2(\R^d;\d x)$.
We note here that the constant $c_3$ above is independent of $C$,
but both $C_0:=C_0(C)$ and $R_0:=R_0(C)$ depend on $C$.

Now, for any $r>0$, we first take $C= 1/r$
(i.e., $1/C= r$ and $C_0:=C_0(C)=C_0(1/r)$ is fixed),
and then choose $s:=s(r)=r/(2c_3(1+rC_0(1/r)))>0$
so that $c_3 s(1+C_0/C)=r/2$. Then,
for any $f\in C_c^\infty(\R^d)$
and for any strictly positive function $\psi\in L^2(\R^d;\d x)$,
\begin{align*}\int_{\R^d} f(x)^2\,\d x
\le & r\E(f,f)+\frac{r}{2}\int_{\R^d} f(x)^2\,\d x \\
&
+\alpha(r)\left(\sup_{z\in B(0,R_0(1/r)+1)}\psi(z)^{-1}\right)
\left(\int_{\R^d} |f(x)|\psi(x)\,\d x\right)^2,\end{align*}
where
$$\alpha(r):= \beta_0(s(r))\left(1+ {rC_0(1/r)} \right).$$
This in particular implies that for all
$0<r\le 1$,
$$\int_{\R^d} f(x)^2\,\d x \le 2r\E(f,f)
+2\alpha(r)\left(\sup_{z\in B(0,R_0(1/r)+1)}\psi(z)^{-1}\right)
\left(\int_{\R^d} |f(x)|\psi(x)\,\d x\right)^2.$$ Hence, for all $0<r\le 2$,
$$\int_{\R^d} f(x)^2\,\d x \le  r\E(f,f)
+2\alpha(r/2)\left(\sup_{z\in B(0,R_0(2/r)+1)}\psi(z)^{-1}\right)
\left(\int_{\R^d} |f(x)|\psi(x)\,\d x\right)^2.$$
Therefore, for any $r>0$ and $f\in C_c^\infty(\R^d)$,
\begin{equation}\label{e:super-}
\int_{\R^d} f(x)^2\,\d x
\le r\E(f,f)+\beta(r)\left(\int_{\R^d} |f(x)|\psi(x)\,\d x\right)^2,
\end{equation} where
$$\beta(r)= 2\alpha((2\wedge r)/2))
\left(\sup_{z\in B(0,R_0(2/(2\wedge r)+1)}\psi(z)^{-1}\right).$$
Note that, in \eqref{e:super-} we can take any strictly positive and bounded function $\psi\in L^2(\R^d;\d x)$; for example, $\psi(x)=(1+|x|)^{-d-\theta}$ with $\theta>0$
or $\psi(x)=e^{-c|x|^\theta}$ with $c,\theta>0$. In particular, with this choice, $\beta(r)<\infty$ for all $r>0$.
Since $C_c^{\infty}(\R^d)$ is $\|\cdot\|_{{\E}_1}$-dense in $\F$,
\eqref{e:super-} is valid for any $f\in \F$.
Hence, we obtain the essential super Poincar\'e inequality for $(\E,\F)$.
Combining this with \cite[Theorems 3.1.7, 3.2.1 and 0.3.9]{W05}, we prove that
the semigroup $(P_t)_{t\ge0}$ is compact on $L^2(\R^d;\d x)$, also thanks to the assertion claimed above that $(P_t)_{t\ge0}$ has a transition density function
with respect to the Lebesgue measure.
\end{proof}

\begin{rem}\rm As seen from the proof above, Proposition \ref{prop:ratio} is the key ingredient to yield Theorem {\rm \ref{thm:cpt}}. This along with Remark \ref{e:remark-key} partly explain the reason why the criteria for compactness of the semigroup associated with the non-local Dirichlet form $(\E^*,\F^*)$ on $L^2(\R^d;\d x)$ with finite range jumping kernel is similar to these for the local Dirichlet form given by
\eqref{eq:elliptic}. On the other hand, following the proof of Theorem \ref{thm:cpt} (in particular the comparison argument),  one may show that
the semigroup associated with the local Dirichlet form
$(\E^L,\F^L)$ given in Remark \ref{rem:diff} is compact, if
$$
\liminf_{|x|\to \infty}\inf_{\xi\in \R^d}\frac{\sum_{i,j=1}^d a_{ij}(x)\xi_i\xi_j}{|x|^2|\xi|^2}=\infty.
$$
\end{rem}

\section{Proof of Theorem \ref{T:main} and further examples}\label{section3}
We first present the
\begin{proof}[Proof of Theorem $\ref{T:main}$]

Let $W(x,y)$ be as in \eqref{eq:funct-w}
such that $U_i$ $(i=1,2)$ is locally bounded
and \eqref{e:cond000} holds.
Then it is easy to verify the condition  \eqref{e:core}.

We first prove (i).
Under the setting in this assertion,
$W(x,y)$ satisfies Assumption \ref{assum:w}(i).
Let $J(x,\d y)=|x-y|^{-d-\alpha}\,\d y$
and $\nu(\d z)={\bf 1}_{\{0<|z|<1\}}|z|^{-d-\alpha}\,\d z$.
Then, the measure $\nu$ is symmetric, and
$J(x,x+A)=\nu(A)$ for any $x\in \R^d$
and for any Borel set $A\subset B(0,1)$.
Moreover, by the proof of \cite[Proposition 2.2]{CKK08},
there exists $c_1>0$ such that for any $\xi\in \R^d$,
\begin{equation}\label{eq:symbol}
\int_{\{0<|z|<1\}}(1-\cos\langle z,\xi\rangle)\,\nu(\d z)
\geq c_1(|\xi|^2\wedge |\xi|^{\alpha}).
\end{equation}
Hence Assumption \ref{assum:w}(ii) is fulfilled
with $\beta_0(r)=c_2(r^{-d/\alpha}\vee r^{-d/2})$ for all $r>0$.
Then by Theorem \ref{thm:cpt},
the semigroup $(P_t)_{t\ge0}$ is compact on $L^2(\R^d;\d x)$.

We next prove (ii).
Under the setting in this assertion,
by some simple calculations, we find that for all $l\ge1$,
$$\int_{\{|x|\le l\}}
 \int_{\R^d}
\left(1\wedge \frac{|x-y|^2}{l^2}\right){W(x,y)}\,J(x,\d y)\,\d x
\le c_3(l^{d}+l^{d+p-\alpha}). $$
This along with
the assumption $p\in [0,\alpha)$ yields that \eqref{e:sss} holds true. Therefore,
the semigroup $(P_t)_{t\ge0}$ is not compact on $L^2(\R^d;\d x)$
by Theorem \ref{Thm2.1}.
\end{proof}

We next present further examples of non-local Dirichlet forms to which
Theorems \ref{Thm2.1} and  \ref{thm:cpt}
are applicable.

Let $\sigma(d\theta)$ be a finite and symmetric nonnegative measure
on the unit sphere $\Ss$ in $\R^d$. Assume that
$\sigma$ is
nondegenerate in the sense that its support is not contained in any proper linear subspace of $\R^d$.
For any $\alpha\in (0,2)$,
let
\begin{equation}\label{e:meausure}
\nu(A)=\int_0^1
\left(\int_{\Ss}{\bf 1}_A(r\theta) \,\sigma(\d \theta)\right)
r^{1+\alpha}\,\d r,
\quad A\in \mathcal{B}(\R^d).
\end{equation}
Then, by
the proof of \cite[Example 1.5]{SSW12}, we know that \eqref{eq:symbol} still holds true.
With this fact at hand, we have the following two examples, which indicate that Theorems \ref{Thm2.1} and  \ref{thm:cpt} work for some degenerate or singular
(with respect to the Lebesgue measure) jumping kernels.

\begin{exam}\label{ex:dege}\rm
Let $\Lambda$ be an infinite cone on $\R^d$ with $d\ge2$ that has non-empty interior and is symmetric with respect to the
origin, i.e.,
$\lambda x\in \Lambda$ if $x\in \Lambda$ and $\lambda\in \R$.
For $\alpha\in (0,2)$, let
$$
\nu(\d z)
= {\bf1}_{\{
0<
|z|<1\}}\frac{1}{|z|^{d+\alpha}}{\bf1}_{\Lambda}(z)\,\d z
$$
and
$
J_1(x,\d y)=\nu(\d (y-x))$.
Then,
$J_1(x,\d y)$ satisfies Assumption \ref{assum:w}(ii) with the measure $\nu$.
On the other hand, \eqref{e:meausure} holds with $\sigma(\d \theta)={\bf 1}_{\Lambda\cap \Ss}(\theta)\,\mu(\d \theta)$, where $\mu$ is the Lebesgue surface measure on $\Ss$. Hence, \eqref{eq:symbol} remains valid for $\nu$.
Therefore, according to the proof of Theorem \ref{T:main},
the assertion of Theorem \ref{T:main} is true
with
$J(x,\d y)=|x-y|^{-d-\alpha}\,\d y$
replaced by
the jumping kernel $J_1(x,\d y)$ above.

In fact, following the proof of Theorem \ref{T:main}, we can obtain that Theorem \ref{T:main} also holds true for
$$
J_2(x,\d y)
=\left(\sum_{n=0}^{\infty}|x-y|^{-(d+\alpha)} {\bf 1}_{\{|x-y|\in [2^{-(2n+1)}, 2^{-2n})\}}\right)\,\d y,
$$ with $\alpha\in (0,2)$.
\end{exam}

\begin{exam}\rm
For any $\alpha\in (0,2)$,
let
$$\nu(\d z)=\sum_{i=1}^d \left[{\bf 1}_{\{
0<
|z_i|< 1\}}\frac{\d z_i}{|z_i|^{1+\alpha}}\otimes \prod_{j\neq i} \delta_{0}(\d z_j)\right],$$ where $\delta_0$ is the Dirac measure at $0\in \R$.
Let $\{e_i\}_{i=1}^d$ be the standard orthonormal basis of $\R^d$
and
$$
\sigma(\d \theta):=\sum_{i=1}^d \delta_{\pm e_i}(\d \theta),
$$
where $\delta_{\theta_0}(\d \theta)$ denotes
the Dirac measure on $\Ss$ concentrated at $\theta_0\in \Ss$.
Then the measure $\nu$ satisfies \eqref{eq:symbol}
because \eqref{e:meausure} holds with this $\sigma$.
Therefore, by the proof of Theorem \ref{T:main},
the assertion of Theorem \ref{T:main} is true
with
$J(x,\d y)=|x-y|^{-d-\alpha}\,\d y$
replaced by
the jumping kernel $J_3(x,\d y)=\nu(\d (y-x))$.

Indeed, we can further extend the example above to more general setting.
Let $d=d_1+d_2$ with $d_i\ge1$ for $1\le i\le 2$.
For $i=1,2$, let $\delta_z^{(i)}$ be the $d_i$-dimensional
Dirac measure at $z\in \R^{d_i}$,
and $m^{(i)}$ the $d_i$-dimensional Lebesgue measure.
For $\alpha_1,\alpha_2\in (0,2)$,
define a rotationally invariant measure $\nu$ on $\R^{d_1+d_2}$ by
\begin{align*}
\nu(\d z)
&=\nu(\d z_1\times \d z_2)\\
&={\bf 1}_{\{
0<
|z_1|< 1\}}\frac{1}{|z_1|^{d_1+\alpha_1}}m^{(1)}(\d z_1)\otimes \delta_{0}^{(2)}(\d z_2)\\
&\quad+{\bf 1}_{\{
0<
|z_2|< 1\}}\frac{1}{|z_2|^{d_2+\alpha_2}}\delta_{0}^{(1)}(\d z_1)\otimes m^{(2)}(\d z_2),
\end{align*}
and  the kernel $J_4(x,\d y)$ on $\R^{d_1+d_2}\times {\cal B}(\R^{d_1+d_2})$ by
\begin{align*}
J_4(x,\d y)
&=\nu(\d (y -x))\\
&={\bf 1}_{\{
0<
|x_1-y_1|< 1\}}\frac{1}{|x_1-y_1|^{d_1+\alpha_1}}\,m^{(1)}(\d y_1)\delta_{x_2}^{(2)}(\d y_2)\\
&\quad
+{\bf 1}_{\{
0<
|x_2-y_2|< 1\}}\frac{1}{|x_2-y_2|^{d_2+\alpha_2}}\,\delta_{x_1}^{(1)}(\d y_1)m^{(2)}(\d y_2).
\end{align*}
Note that the measure $\nu(\d z)$ above is
 a L\'evy measure of the direct product
of
truncated
symmetric $\alpha_i$-stable processes on $\R^{d_i}$ with $i=1,2$. Then, we can claim that the assertion of Theorem \ref{T:main} remains true
with $q\in [0,\alpha_1\wedge \alpha_2)$
and the jumping kernel $J_4(x,\d y)$ above.
\end{exam}

\noindent \textbf{Acknowledgements.}
The research of Yuichi Shiozawa is supported
in part by JSPS KAKENHI No.\ JP17K05299. The research of Jian Wang is supported by the National Natural Science Foundation of China (No.\ 11831014),
the Program for Probability and Statistics:
Theory and Application (No.\ IRTL1704),
and the Program for Innovative Research Team in Science and Technology
in Fujian Province University (IRTSTFJ).

\end{document}